\documentclass{article}
\usepackage{pgfplots}
\usepackage{fullpage}
\usepackage{amsmath}
\usepackage{amsthm}
\usepackage{amssymb}
\usepackage{url}
\usepackage{graphicx}
\usepackage{verbatim}
\usepackage{url}
\usepackage{graphicx}
\usepackage[caption=false]{subfig}
\usepackage{float}
\usepackage[ruled]{algorithm2e}
\usepackage{tikz}
\usetikzlibrary{calc}
\usetikzlibrary{patterns}
\usepackage{tikz}
\tikzstyle{mybox} = [draw=black, fill=white,  thick,
rectangle, inner sep=10pt, inner ysep=20pt]
\tikzstyle{mybox} = [draw=black, fill=white,  thick,
rectangle, inner sep=2pt, inner ysep=2pt]
\usetikzlibrary{arrows}
\usetikzlibrary{arrows,chains,matrix,positioning,scopes}
\usetikzlibrary{arrows,shapes,positioning}
\usetikzlibrary{decorations.markings}
\tikzstyle arrowstyle=[scale=1]
\tikzstyle directed=[postaction={decorate,decoration={markings,
		mark=at position .65 with {\arrow[arrowstyle]{stealth}}}}]
\tikzstyle reverse directed=[postaction={decorate,decoration={markings,
		mark=at position .65 with {\arrowreversed[arrowstyle]{stealth};}}}]
\newcommand{\boundellipse}[3]% center, xdim, ydim
{(#1) ellipse (#2 and #3)
}
\newtheorem{theorem}{Theorem}
\newtheorem{lemma}{Lemma}
\newtheorem{cor}{Corollary}
\newtheorem{prop}{Proposition}
\theoremstyle{definition}
\newtheorem{definition}{Definition}
\newtheorem{remark}{Remark}

\usepackage{multirow}
\usepackage[utf8]{inputenc}

\begin{document}

%\title{On Tusi's Method and Connections with Cardano's Formula}

\title{On Tusi's Classification of Cubic Equations and its Connections to  Cardano's Formula and  Khayyam's Geometric Solution}

\author{Bahman Kalantari\footnote{Emeritus Professor of Computer Science, Rutgers University, Piscataway, New Jersey, U.S.A} and Rahim Zaare-Nahandi\footnote{Emeritus Professor of Mathematics, University of Tehran, Tehran, Iran}}
\date{}
\maketitle

\begin{abstract}
Omar Khayyam's studies on cubic equations inspired the 12th century  Persian mathematician Sharaf al-Din Tusi to investigate the number of positive roots.  According to the translation of Tusi's work by the noted mathematical historian Rashed, Tusi analyzed the problem for five different types of equations. In fact all cubic equations are reducible to a single canonical form we call {\it Tusi form} $x^2-x^3=c$. Tusi determined that the maximum of $x^2-x^3$ on $(0,1)$ occurs at $\frac{2}{3}$ and concluded when $c=\frac{4}{27} \delta$, $\delta \in (0,1)$, there are roots in $(0, \frac{2}{3})$ and $(\frac{2}{3},1)$,  ignoring the root in  $(-\frac{1}{3},0)$. Given a {\it reduced form} $x^3+px+q=0$, when $p <0$, we show it is reducible to a Tusi form with $\delta = \frac{1}{2} +  {3\sqrt{3} q}/{4\sqrt{-p^3}}$. It follows that there are three real roots if and only if the discriminant $\Delta =-(\frac{q^2}{4}+\frac{p^3}{27})$ is positive. This gives
an explicit connection between $\delta$ in Tusi form and $\Delta$ in Cardano's formula. Thus when $\delta \in (0,1)$, rather than using Cardano's formula in complex numbers one can use the intervals in Tusi form to approximate the roots iteratively. On the other hand, for a reduced form with $p >0$ we give a novel proof of Cardono's formula.   While Rashed  attributes Tusi's computation of the maximum to the use of derivatives, according to Hogendijk, ``Tusi probably found his results by means of manipulation of squares and rectangles on the basis of Book II of Euclid's {\it Elements}.''  Here we show the maximizer in Tusi form
is directly computable via elementary algebraic manipulations. Indeed for a {\it quadratic Tusi form}, $x-x^2=\delta/4$, Tusi's approach results in a simple derivation of the quadratic formula, comparable with the pedagogical approach of Po-Shen Loh. Moreover, we derive analogous results for the {\it general Tusi form}, $x^{n-1}-x^n=\delta (n-1)^{(n-1)}/n^n$. Finally, gaining insights from Tusi form, we present a concise derivation of Khayyam's geometric solution for all cubic equations. The results here complement previous findings on Tusi's work and reveal further facts on
history, mathematics and pedagogy in solving cubic equations.
\end{abstract}


\section{Introduction} \label{sec1} In this article we examine the problem of solving for the real roots of a cubic equation
but in the context of the work of the 12th century Persian mathematician Sharaf al-Din Tusi.  In the process we offer novel insights, including canonical representation of cubic equations, strategies for the approximation of the real roots, connections between the work of Tusi and Cardano's formula discovered several centuries later, as well as connections to Omar Khayyam's geometric solution.
To summarise, we believe the article bridges three significant historic works on cubic equations: Khayyam's geometric solution, Tusi's classification and Cardano's formula, turning these connections into relevant subjects of study in today's curricula.

It is now well known that any cubic equation with real coefficients via affine transformation
can be written in the {\it reduced form} (also called {\it depressed form}):
\begin{equation}
x^3+px+q=0, \quad p, q \in \mathbb{R}.
\end{equation}
In fact  we may assume $p=0, \pm 1$. The case of $p=0$ is trivial. When $p= \pm 1$ we refer to the equation as the {\it normal form}. Any reduced form must have at least one real root, a consequence of the {\it Intermediate Value Theorem}, a fact that most likely was intuitively known at the time of Tusi and even sooner. Given the reduced form, {\it Cardano's formula} for the roots  takes the form
$$
\bigg (-\frac{q}{2} + \sqrt{-\Delta} \bigg )^{1/3} + \bigg (-\frac{q}{2} - \sqrt{-\Delta} \bigg )^{1/3}, \quad \Delta = - \bigg (\frac{q^2}{4} + \frac{p^3}{27} \bigg ).$$

Cardano's formula, apparently first due to del Ferro (1465-1562), was finally published by Cardano in 1545. Others credited to the formula are, Tartaglia and Ferrari.  For an account of the history of cubic equations see e.g., Irving \cite{Irving}.  The discovery of the formula for the solutions of a cubic equation is no doubt very significant by itself, even if one ignores all other consequences, such as the discovery of complex numbers and ultimately Galois theory.  Nevertheless, it does not imply that a formula is always superior to a numerical approximation methods.  This is obvious even in solving a quadratic equation.  A closed formula for roots of $x^2-2=0$ does not provide numerical approximation to $\pm \sqrt{2}$.  But in the case of Cardano's formula other oddities can happen.
When there are three distinct real roots  the {\it discriminant} $\Delta$ is positive so that the formula necessarily passes to complex numbers, a fact that led to the invention of complex number. In this case, if there is no rational root, it is impossible to express the roots in real radicals, see \cite{Warden}.  In other words Cardano's formula is necessarily expressed in terms of the cube roots of two complex numbers. This according to Turnbull \cite{Turn}, see also Zucker \cite{Zucker}, is ``paradoxical.'' He goes on to say,

``{\it To mathematicians of the 16th and 17th centuries this feature was mysterious: they spoke of the {\it irreducible} case. DeMoivre's theorem allows us to negotiate a calculation of the roots; but it remains a curious fact that from a real cubic three real roots cannot be extracted by Cardan's algebraic formula without a circuitous passage into, and out of, the domain of complex numbers.}''

Using DeMoivre's theorem we can avoid the complex plane but at the cost of using transcendental functions with transcendental arguments. Zucker \cite{Zucker} gives a way to bypass the use of trigonometric functions but at the cost of using hypergeometric functions.

Whether we consider solving a cubic equation from  historical point of view,  mathematical, philosophical, practical, or pedagogical point of view, it is undoubtedly a fascinating subject. In this article we refer to the work of Tusi and show it  connects with the work resulting in Cardono's formula and  complements
it in several ways, revealing interesting facts about cubic equations.  The work of Sharaf al-Din Tusi on cubic equations has been analyzed and documented in detail by the noted historian of mathematics of the Golden Age of Islam, Roshdi Rashed \cite{Rashed1, Rashed2}. In this article we rely on Rashed's work and build on it.
Tusi's findings followed the work of Omar Khayyam on cubic equations, both on classifying the number of  positive zeros of a cubic equation and determining some roots. Negative numbers were avoided at the time. Also, equations were written
in such a way that would make the coefficients positive.  In so doing, according to Rashed, Khayyam had considered $25$ forms of polynomials of degree at most three. Quadratic equations were classified by Khwarizmi \cite{Khar} several centuries earlier.  For a history of quadratic and cubic equations see e.g., Katz \cite{Katz} and Irving \cite{Irving}. Khayyam also had derived geometric solutions for positive roots of cubic polynomials. Tusi's treatise in the first part covers $20$ forms of the $25$ polynomials with more rigor than Khayyam. But for the last $5$ forms he analyzes whether or not there are positive roots by means of computation of maximum of the relevant polynomials when the constant term is discarded. The five cubic equations Tusi considers, all with positive coefficients, were,  $x^3 + c = ax^2$, $x^3 + c = bx$, $x^3 + ax^2 +c = bx$,
$x^3 + bx +c = ax^2$, and $x^3 +c = ax^2 + bx$. While mathematical notations where not invented, they could manipulated symbols that  represented these unknowns. Indeed already Khayyam (even Khwarizmi) had names for the unknown $x$ as well as for $x^2$, $x^3$, etc. The fact that they did not allow negative coefficients does not mean they did not deal with negative numbers implicitly. After all, terms with negative coefficients could be moved to the other side of the equation.

The first of the five cubic e1quations Tusi considered, $x^3 + c = ax^2$, can be reduced to the form $x^3+c=x^2$. According to Rashed's examination, in determining the number of positive roots Tusi considered the roots of $x^2 - x^3$ and its maximum value, which Tusi calls ``the greatest number'', by looking at the roots of its derivative, without giving a name for it. Rashed claims Tusi argued the maximum occurs when $x = 2/3$ and concluded the equation $x^2 - x^3 = c$ has exactly one positive solution when $c = 4/27$, and two positive solutions when $0 < c < 4/27$, one in the intervals $(0, \frac{2}{3})$ and another in $(\frac{2}{3},1)$. A third root  in  $(-\frac{1}{3},0)$ was not considered. Also, the case of negative $c$ was not considered. Rashed emphasizes this is not the first time Tusi uses the concept of derivative and that he had already used it for numerical solutions of equations. Rashed's conclusion on the use of derivative has been contested by other scholars who argue that Tusi could have obtained the result on the number of positive roots by other methods, not requiring derivatives, see for example Berggren \cite{Berg}.  Furthermore, Houzel \cite{Houz} has given a different possibility for the use of derivative by Tusi, even in the treatment of numerical solution of equations.  Indeed, Houzel has discussed several examples given by Tusi with his algorithm for finding a positive integer solution of a cubic equation using a polygon which could somehow resemble Newton's polygon for a polynomial in two variables (see \cite{Christ} for an account on Newton's method for approximating roots and resolution of affected equations).  Nevertheless, Rashed's translation of Tusi's work is universally praised.  For example, Hogendijk \cite{Hogen} considers  Rashed's translation significant, saying

\vspace{.25cm}

``{\it Until recently Omar Khayyam was supposed to have given the most advanced medieval treatment of cubic equations. Thanks to Rashed's publication \cite{Rashed1, Rashed2} we now know al-Tusi went considerably further.}''

\vspace{.25cm}

Yet Hogendijk suggests that since cubic curves were never drawn by medieval mathematicians, nor derivatives  mentioned explicitly in
any known medieval Arabic text, the question arises whether Tusi's
methods and motivation can also be explained in terms of standard ancient and
medieval mathematics. He proposes such an alternative explanation and concludes,

\vspace{.25cm}

``{\it Tusi probably found his results by means of manipulations of squares and rectangles on the basis of Book II of Euclid’s Elements.}''

\vspace{.25cm}

Our initial interest was to see if we could find an easy argument using the knowledge of the time that would verifiably and unambiguously compute the maximum value arising in Tusi's analysis. The fact that we are not historians of math does not preclude us from such an investigation, even if we restrict it to the analysis of historians on Tusi's work.  We were able to conclude that the maximizer in Tusi form can be computed via elementary algebraic manipulation that were possible at Tusi's time.  Even geometric-based arguments and manipulations require working with inequalities and equalities so that even if modern notations were not present at Tusi's time, algebraic manipulations were possible.

On the one hand, we show how to directly compute the maximizer. On the other hand, we argue it is easy to guess the location of the maximizer and then verify whether or not a guessed value is the correct maximizer.  Using the maximizer is one component in Tusi's characterization of positive zeros.  A second component that cannot be ignored is that even if in Tusi's time the notion of real numbers was not present, an intuitive notion of the {\it Intermediate Value Theorem} was necessary to conclude the existence of roots in the designated intervals.  In this sense geometric arguments alone are not sufficient.  More importantly, in our investigations we were able to derive further results, including an explicit connection between Tusi's classification and Cardano's formula discovered  centuries later, and even a motivation to give a novel derivation of this formula, as well as a concise and simple  derivation of Khayyam's geometric solution for all cubic equations.  These finding also give rise to algorithms for approximations of  the roots of a cubic equation. In fact by extending Tusi form, we gain insights on the quadratic formula and even high degree Tusi forms.  In summary, our results complement previous findings on Tusi's work and show novelties about solving cubic equations that will be of interest, not just to historians of mathematics but the community of mathematicians, mathematics educators and students.

Studying Tusi's work gives rise to classification of all cubic equations with real coefficients taking one of the two exclusive forms
\begin{equation}
x^3-x^2+\frac{4}{27}\delta =0, \quad \delta \in (0,1); \quad {\rm ~~and~~}
x^3 + x+q=0.
\end{equation}
We shall refer to the first equation without any restrictions on $\delta$  as {\it Tusi form}. In this article we show every reduced cubic equation with $p <0$ is in turn reducible to a  Tusi form with $\delta = \frac{1}{2} +  \frac{3\sqrt{3}q}{4 \sqrt{-p^3}}$ and then from Tusi's classification it follows that it has three real roots if and only if $\Delta >0$.  In particular, characterization of zeros of Tusi form gives rise to the use of $\Delta$ present in Cardano's formula. When the discriminant is positive, Cardano's formula is necessarily expressed in complex numbers. However, Tusi form with $\delta \in (0,1)$ necessarily will have a root in each of the three intervals, $(0,2/3)$, $(2/3,1)$ and $(-1/3, 0)$.  This information can be used to approximation the roots, e.g., by using the  bisection method or  Newton method.  Indeed even a precomputed table of discrete  values of $x^2-x^3$ in the interval $[-1/3,1]$ (an interval containing the three roots) to a desired precision suffices in the approximation of the three roots of a Tusi form with $\delta \in (0,1)$. In summary, Tusi form simplifies solving a cubic equation by their formulation into a corresponding Tusi form, $x^3-x^2+4\delta/27=0$, where there are three distinct real roots if and only if $\delta \in (0,1)$.
The normal form $x^3+x+q=0$ covers all other cases. We may conclude a real cubic has a single real root if and only if it is either reducible to Tusi form with $\delta \not \in (0,1)$, or reducible to a reduced form with $p=1$.

The remaining sections are organized as follows. In Section \ref{sec2} we consider Tusi form and characterize its zeros. In Section \ref{sec3} we characterize the zeros of reduced and normal forms.  In Section \ref{sec4} we first derive Cardano's formula for normal forms and then extends it to reduced forms.  In Section \ref{sec5} we derive Khayyam's geometric solution of cubic equation for normal forms and describe an algorithm for approximation of the roots utilizing this formulation.  In Section \ref{sec6} we consider solving quadratic equations, showing  all real quadratic equations are reducible to the following {\it quadratic Tusi form},
\begin{equation}
x^2-x+ \frac{\delta}{4}=0, \quad \delta \in \mathbb{R}.
\end{equation}
This trivially leads to the quadratic formula.  Knowing how to solve this single quadratic equation is enough to solve  all quadratic equations. We contrast this approach to solving a quadratic equation with Po-Shen Loh's \cite{Loh} derivation of the quadratic formula (also featured in The New York Times, see \cite{NYT}), considered novel and pedagogically advantageous to the standard method based on completing the square. In Section \ref{sec7} we extend relevant results to the {\it general Tusi form}
\begin{equation}
x^{n}-x^{n-1}+ \frac{{(n-1)}^{n-1}}{n^n} \delta =0, \quad n \geq 2, \quad \delta \in \mathbb{R}
\end{equation}
while using  much the same reasoning as the case of cubic equation. We end with brief concluding remarks.

\section{Tusi Form and Characterization of its Zeros} \label{sec2} Tusi considered the number of positive  zeros of
\begin{equation}
p(x)=x^3-bx^2+c=0,   \quad b >0.
\end{equation}
He assumed $c >0$, however we consider $c$ to be arbitrary. The above equation is equivalent to
\begin{equation}
g(x)= bx^2 - x^3= x^2(b-x)=c.
\end{equation}
Graphically, the problem seeks the number of intersections of the horizontal line $y=c$ and the curve defined by $g(x)$. Clearly, $q(x)=0$ at $x=0$ and $x=b$, while it stays positive on $(0,b)$ and  negative on $(0, \infty)$.  Since negative zeros were ignored at the time, the focus of Tusi's search was the interval $(0,b)$. Apparently Tusi reasoned that in order to determine how many solutions are possible for a given $c$, he had to determine the largest value of $g(x)$ on the interval $(0,b)$. Using derivatives, trivially the maximum occurs at $2b/3$ and is $4b^3/27$. Moreover, $g(x)$ is monotonically increasing on $(0,2b/3)$, and decreasing on $(2b/3,b)$. Thus for any $c \in (0,4b^3/27)$ there are two positive solutions to the cubic equation. The details of Tusi's reasoning, as interpreted by  Rashed, is given in \cite{Rashed2}.  We offer a new interpretation. We can reformulate Tusi's cubic by replacing $x$ with $\alpha b$, reducing it into a canonical form. Letting
\begin{equation}
\phi(\alpha) = \alpha^{2}- \alpha^3,
\end{equation}
then  $g(\alpha b)= b^3 \phi(\alpha)$.
Thus  $g(\alpha b)=c$ if and only if $\phi(\alpha)= c/b^3$ if and only if $p(\alpha b)=0$.

To solve $\phi(\alpha) = constant$ for positive solutions, we observe that $\phi$ is zero at $\alpha=0, 1$.  Thus if we can determine the location of the maximizer, $\alpha^*$,  of $\phi(\alpha)$ between zero and one, then by changing the $constant$ value between zero and the maximum value, $\phi^*= \phi(\alpha^*)$, we may conclude that positive roots of $\phi(\alpha) = constant$  must lie in $(0,\alpha^*)$ and $(\alpha^*,1)$.  In Proposition \ref{prop0} we derive $\alpha^*$ via a simple argument without differentiation.
However, one may ask: Can we guess the value of $\alpha^*$? After all, nowadays,  as well as throughout the history, guessing the answer to math problems is one of the significant techniques to solving them. But guessing must be followed by rigorous proof. Even if an initial guess is wrong, trial and error will soon lead to the correct answer. Note that $\alpha^2-\alpha^3=\alpha^2(1-\alpha)$. Thus  $\alpha^* \geq 1/2$.  One would try $1/2$, $3/4$, $2/3$, etc.  The fact that $\alpha - \alpha^*$ must be a factor of $\phi(\alpha) - \phi(\alpha^*)$ provides the mechanism for verification of the correctness of a guessed value for $\alpha^*$. Lemma \ref{lem3z} proves the maximum value is $4/27$, attained at $2/3$. Thus we set $constant =\delta 4/27$ and when $\delta \in (0,1)$ the solutions to $\phi(\alpha) = constant$ lie in $(0,1)$.

\begin{definition}
We say a cubic equation is in {\it Tusi form} if it is given as $\alpha^3-\alpha^2+\delta \frac{4}{27}=0$, $\delta \in \mathbb{R}$.
\end{definition}

\begin{prop}   \label{prop0} The maximum value of  $\phi(\alpha)$ on $[0, 1]$ is attained at $\frac{2}{3}$.
\end{prop}

\begin{proof} Let $\alpha^*$ be the maximizer. It must be a root of
$$h(\alpha)=\phi(\alpha)-\phi(\alpha^*)=(\alpha^2-{\alpha^*}^2) - (\alpha^3-{\alpha^*}^3).$$
But $h(\alpha)$ can be written as
$$h(\alpha)=-(\alpha-\alpha^*)q(\alpha), \quad q(\alpha)=\alpha^2+(\alpha^* -1)\alpha+\alpha^*(\alpha^*-1).$$
Now $q(0)=\alpha^*(\alpha^*- 1) <0$, $q(1)={\alpha^*}^2 >0$ so that $h(\alpha)$ has another root $\beta^* \in (0,1)$. Solving $q(\alpha)=0$ via the quadratic formula it is easy to see there is also a negative root.
Thus if $\alpha^*=\beta^*$ we get
$$\sqrt{(\alpha^*-1)^2-4 \alpha^*(\alpha^* -1)}=(3 \alpha^* -1).$$
Squaring and solving for $\alpha^*$, we get $\alpha^* =2/3$. Suppose $\alpha^* \not = \beta^*$. We argue
$\phi((\alpha^*+ \beta^*)/2)$ cannot be larger, equal or smaller than  $\phi(\alpha^*)$.  It cannot be larger because $\alpha^*$ is maximizer. It cannot be equal because then $h(\alpha)=0$ has three positive roots but we know it has a negative root.  If it is less, then for any value $c$ between the $\phi((\alpha^*+ \beta^*)/2)$ and $\phi(\alpha^*)$ then because of fluctuation of $\phi$ the equation  $\phi(\alpha)=c$ must have more than three roots in $(0,1)$, a contradiction.
\end{proof}

\begin{remark} One may claim that Tusi did not have algebraic notations and for him $x^3$ and $x^2$ were solids. However, even so he must have manipulated symbols that  represented these unknowns. Indeed already Khayyam (even Khwarizmi) had names for the unknown $x$ as well as for $x^2$, $x^3$, etc.
In fact Hogendijk  who thinks  Tusi ``probably'' found his results by means of manipulation of squares and rectangles on the basis of Book II of Euclid's Elements, still argues (pages 74,75) that Tusi defines the maximizer algebraically.  Though the true derivation of the value of maximizer is unclear.     Note that  from Proposition \ref{prop0} the derivation of the maximizer does  not automatically follow from factorization of $\phi(\alpha)-\phi(\alpha^*)$. Additional arguments are needed for it.   In this sense Tusi's derivation needs more than Euclid's geometric arguments to derive the maximizer.   As will be shown in the next lemma,  another approach starts by guessing the value for the maximizer.  Either approach of Proposition \ref{prop0} or Lemma \ref{lem3z} requires factorization as a first step in the proof, a process that was certainly possible at Tusi's time, even in the absence of formal mathematical notations.
\end{remark}

\begin{lemma}   \label{lem3z} The maximum value of  $\phi(\alpha)$ on $[0, \infty)$ is $\phi^*=\frac{4}{27}$, attained at $\alpha^*= \frac{2}{3}$.
Moreover,

$$
\phi(\alpha)
\begin{cases}
\rightarrow \infty, & \text{as~} \alpha  \rightarrow -\infty\\
\nearrow   \text{on~} [0, \frac{2}{3}]\\
\searrow  \text{on~} [\frac{2}{3}, \infty)
\end{cases}
$$
where $\nearrow$ and $\searrow$ mean strictly increasing and decreasing, respectively.
\end{lemma}

\begin{proof} Clearly $\phi(\alpha)$  approaches $+\infty$ as $\alpha$ approaches $-\infty$.  Dividing $\phi(\alpha) - \phi^*$  by $(\alpha -2/3)$, twice, it is straightforward to show
\begin{equation} \label{eqcub1}
\alpha^2-\alpha^3-\frac{4}{27}=-(\alpha - \frac{2}{3})^2(\alpha +\frac{1}{3}).
\end{equation}
For $\alpha > 2/3$ the above is strictly decreasing.   When $0 \leq \alpha_1 <\alpha_2 \leq 2/3$,
$-(\alpha_1 - \frac{2}{3})^2 < - (\alpha_2 - \frac{2}{3})^2$. From this inequality and (\ref{eqcub1}) it follows that $\alpha^2-\alpha^3-\frac{4}{27}$ is strictly increasing on $[0,2/3]$.
\end{proof}

\begin{remark}  While in Tusi's time formal mathematical notations where not in use, he was aware of what is now called Horner or Ruffini-Horner method which can be used to factor a cubic equation. Historians who have studied Tusi's work have spoken of his knowledge of the method.  Thus the above proof was possible at his time.
\end{remark}

\begin{prop} \label{prop11} {\rm (Intermediate Value Theorem)} If $\phi(\alpha_1) = c_1$ for some $\alpha_1$ and $c_1$ and
$\phi(\alpha_2) = c_2$ for some $\alpha_2$ and $c_2$, $c_1 < c_2$, then for any $c' \in (c_1,c_2)$, there exists $\alpha'$ such that $\phi(\alpha')=c'$. \qed
\end{prop}

\begin{remark} While the formalization of Intermediate Value Theorem uses calculus and continuity, it is intuitive and must have been taken as granted at Tusi's time.
\end{remark}

Theorem \ref{Tusithm} below follows from Lemma \ref{lem3z} and shows the status of zeros of a Tusi form.  Figure 1 shows graphs of $\alpha^2-\alpha^3$ and $\alpha^3+\alpha$.  The latter polynomial, to be analyzed in the next section, is complementary to Tusi form in the sense that any cubic equation is reducible, either to $\alpha^3-\alpha^2+ \delta 4/27=0$, $\delta \in (0,1)$, or to $\alpha^3+\alpha+q=0$. The figure thus gives a visual summary for solving a general cubic equation by reducing it to intersecting one of only two curves and the
horizontal line defined by the constant term of the equation.

\begin{theorem} \label{Tusithm} {\rm (Tusi Theorem)} Given a real constant $\delta$, the number of distinct real solutions and the location of the solutions of the equation  $\alpha^3- \alpha^2+\delta \frac{4}{27}=0$ satisfy the following (the roots are designated as $\underline \alpha$, $\alpha^*$, etc.)
$$
\begin{cases}
1, \quad \underline \alpha \in (-\infty, -\frac{1}{3}),& \text{if ~} \delta > 1\\
2, \quad  \alpha^*=\frac{2}{3}, \quad \overline \alpha^*=-\frac{1}{3}, & \text{if ~} \delta =1\\
3, \quad 0 < \alpha_1 < \frac{2}{3} <\alpha_2  <1, \quad \alpha_3 \in (-\frac{1}{3},0),& \text{if ~} 0 < \delta < 1\\
2, \quad \alpha_1=0,  \quad  \alpha_2=1,& \text{if ~} \delta =0\\
1, \quad \overline \alpha  \in (1, \infty), & \text{if ~} \delta < 0.
\end{cases}
$$
\qed
\end{theorem}

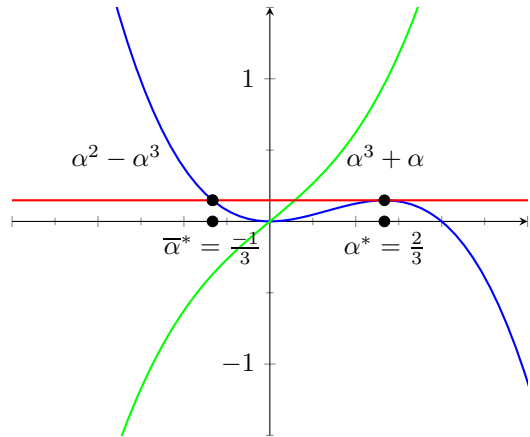
\begin{figure}[htpb] \label{Figone}
\centering
\begin{tikzpicture}[scale=1.]
\begin{axis}[xmin=-1.5,xmax=1.5,ymin=-1.5,ymax=1.5, samples=200,minor tick num=1,xticklabels = {,,},axis lines = middle, legend pos = south east]
\addplot[blue, thick] (x,x*x - x*x*x);
\addplot[green, thick] (x,x*x*x + x);
\addplot[red, thick] (x,4/27);
\addplot[black, mark=*] coordinates {(2/3,0)};
\addplot[black, mark=*] coordinates {(2/3,4/27)} node[below = .3cm] {$\alpha^* = \frac{2}{3}$};
\addplot[black, mark=*] coordinates {(-1/3,4/27)} node[above = .3cm] {~~~~~~~~$\alpha^2-\alpha^3$~~~~~~~~~~~~~~~~~~~~~$\alpha^3+\alpha$};
\addplot[black, mark=*] coordinates{(-1/3,4/27)} node[below = .3cm] {$\overline \alpha^* = \frac{-1}{3}$};
\addplot[black, mark=*] coordinates{(-1/3,0)};
%\addlegendentry{$x^2 - x^3$}
%\addlegendentry{$x^3 + x$}
\end{axis}
\end{tikzpicture}
\caption{Any cubic equation is reducible either to $\alpha^2-\alpha ^3=\frac{4}{27}\delta$, $\delta \in (0,1)$,  or $\alpha^3+\alpha=-q$, $q \in \mathbb{R}$.}
\end{figure}

From the relationship between $\phi$, $g$ and $p$, we get the  more general form of Tusi Theorem

\begin{theorem} \label{TusiGthm} Given the real cubic equation $x^3-bx^2+c=0$, $b >0$, the number of distinct real solutions and the location of the solutions satisfy
$$
\begin{cases}
1, \quad \underline x \in (-\infty, -\frac{b}{3}),& \text{if ~} c > \frac{4b^3}{27}\\
2,  \quad x^*=\frac{2b}{3}, \quad \overline x^*=-\frac{b}{3},  & \text{if ~} c=\frac{4b^3}{27}\\
3,  \quad 0 < x_1 < \frac{2b}{3} < x_2 < b, \quad
x_3 \in (-\frac{b}{3},0),  & \text{if ~} 0 < c <\frac{4b^3}{27}\\
2, \quad x_1=0,  \quad  x_2=b,& \text{if ~} c =0\\
1, \quad \overline x \in (b, \infty), & \text{if ~} c <0.
\end{cases}
$$
\qed
\end{theorem}

\begin{remark} We see that using Tusi's methodology one can easily determine the number of real roots of $x^3-bx^2+c=0$ with $b >0$.  Simply by converting this into its Tusi form, not only we can determine the number of real roots as being one or three but also intervals containing the roots (Theorem \ref{TusiGthm}). In the case of $\delta \in (0,1)$  there are three real roots and for each one an interval containing it is at hand. Using the straightforward bisection method, we can approximate the roots to within $\varepsilon$ precision in $\log 1/\varepsilon$ iterations. To get the corresponding approximations with respect to the original equation, we simply multiply by $b$.  This makes the precision to within $\varepsilon b$.  Indeed given a precomputed table of discrete values for $\alpha^2-\alpha^3$ to within a reasonable precision, we can approximate the roots to within that precision simply by looking up the table for $\alpha$ values corresponding to the given $\delta \in (0,1)$. In a sense when a cubic equation has three roots Tusi's analysis turns the task of approximation of roots into a mechanical process. One can imagine that in Tusi's time approximations up to reasonable precision could have been attained for such cubic equations.  When $\delta \not \in [0,1]$, Tusi's methodology still provides information on where the single root in this case would fall and estimates of the interval containing the root can be easily derived. In the next section we show any cubic equation not reducible to Tusi form with $\delta \in [0,1]$ must be reducible to the form $\alpha^3+\alpha+q=0$ (see Figure 1).
\end{remark}

 \section{Characterization of Zeros of Reduced and Normal Forms} \label{sec3}

\begin{prop} Every nontrivial cubic equation can be written in a reduced form, $x^3+px+q=0$, where $p \not =0$. A reduced form with positive $p$ cannot be reduced to one with negative $p$ and conversely. Moreover, a reduced form with $p \not =0$ can be written in {\it normal form}, i.e. $p=\pm 1$.
\end{prop}

\begin{proof} A cubic equation can be written as $x^3+bx^2+cx+d=0$. Replacing $x$ with $x-t$ results in a cubic equation with the coefficient of quadratic term equal to $(3t+b)$. Thus setting $t=-b/3$ gives a reduced form.  Suppose we have two reduced forms for the same equation, say $x^3+px+q=0$ and $y^3+p'y+q'=0$, where $p >0$ but $p' <0$.  The first cubic function is one-to-one while the second is not (also deducible from Tusi Theorem \ref{Tusithm}).  But such property would be preserved under affine transformations, a contradiction.
To show reduction to normal forms, consider  $x^3+px+q=0$. Replacing $x$ by ${|p|}^{1/2} \beta$ results in the equation $|p|^{3/2}(\beta^3 + {\rm sign} (p) \beta) +q=0$, where sign$(p)$ is the sign of $p$.
\end{proof}

\begin{theorem}  \label{thmdelta} A reduced form $x^3+px+q=0$ with  $p <0$ can be written in Tusi form:
$$
\alpha^3 - \alpha^2+ \delta \frac{4}{27}=0, \quad
\delta = \frac{1}{2} +  \frac{3\sqrt{3}q}{4 \sqrt{-p^3}}.
$$
Conversely, given the Tusi form, its reduced form is
$$
x^3 - \frac{1}{3}x + \frac{4\delta -2}{27}=0.
$$
\end{theorem}
\begin{proof} Replacing $x$ in the reduced form with $x-t$, we get
$$
x^3-3tx^2+(3t^2+p)x-t^3-pt+q =0.
$$
Since $p <0$ we can set $t=\sqrt{\frac{-p}{3}}$ to get
$$
x^3- \sqrt{-3p}x^2+ \frac{2\sqrt{-p^3}}{3\sqrt{3}}+q=0.
$$
Replacing $x$ with $\sqrt{-3p}\alpha$ we get
$$
3 \sqrt{3}\sqrt{-p^3}(\alpha^3- \alpha^2)+ \frac{2\sqrt{-p^3}}{3\sqrt{3}}+q=0.
$$
This gives the first statement of the theorem. Replacing $\alpha$ with $\alpha+\frac{1}{3}$ in a Tusi form proves the converse.
\end{proof}

\begin{cor}  \label{cor1} The equation $x^3+px+q=0$ has three distinct real roots if and only if
$\Delta =-(\frac{q^2}{4}+\frac{p^3}{27}) >0$.
\end{cor}
\begin{proof} From Theorem \ref{Tusithm} there are three distinct real roots if and only if $\delta \in (0,1)$.  Then from Theorem \ref{thmdelta} it follows that $p$ must be negative and the following must hold
$$
0 < \delta = \frac{1}{2} +  \frac{3\sqrt{3}q}{4 \sqrt{-p^3}} <1.
$$
The above is equivalent to
$$
 -\frac{1}{2} < \frac{3\sqrt{3}q}{4 \sqrt{-p^3}} < \frac{1}{2}.
$$
Equivalently,
\begin{equation} \label{corbdd}
\bigg |\frac{3\sqrt{3}q}{4 \sqrt{-p^3}} \bigg | < \frac{1}{2}.
\end{equation}
Squaring the above and moving terms around gives $\Delta >0$.
\end{proof}

\section{Cardano's Formula for Normal and Reduced Forms} \label{sec4}

\begin{theorem} \label{cardano}  Given $x^3+x+q=0$, there exists a positive number $s$ such that the unique real solution $r$ is
$$
r=\bigg (-\frac{q}{2} + s \bigg )^{1/3} + \bigg (-\frac{q}{2} - s \bigg )^{1/3}.
$$
Moreover, $s$ can be expressed explcitly
$$s= \sqrt{\frac{q^2}{4}+ \frac{1}{27}}.$$
\end{theorem}
\begin{proof}
Existence of a real root follows from the Intermediate Value Theorem and its uniqueness follows from the monotonicity of $x^3+x$. Let $A=(-{q}/{2}+ s)^{1/3}$, $B=(-{q}/{2}- s)^{1/3}$. If there is $s >0$ such that $r$ above is a root,
using the identity $(A+B)^3=A^3+B^3 + 3AB(A+B)$,  and $r=A +B$, we get
$$r ^3 + r=
-\frac{q}{2}+ s -\frac{q}{2}- s + 3 \bigg (\frac{q^2}{4}- s^2 \bigg )^{1/3} r + r =- q.$$
This implies $3({q^2}/{4}- s^2)^{1/3} +1=0$. Equivalently
${q^2}/{4}- s^2= -{1}/{27}$. We solve for $s$.
\end{proof}

\begin{cor} \label{cardcor} The reduced equation $x^3+px+q=0$, $p >0$, has a unique real solution
$$
\overline r= \bigg (-\frac{q}{2} + \sqrt{\frac{q^2}{4} + \frac{p^3}{27}} \bigg )^{1/3} + \bigg (-\frac{q}{2} - \sqrt{\frac{q^2}{4} + \frac{p^3}{27}} \bigg )^{1/3}.
$$
\end{cor}

\begin{proof} Replacing $x$ by $\sqrt{p}y$ reduces the equation to the normal form $y^3+y +\frac{q}{p^{3/2}}=0$. Let $r'$ be the solution to the normal form.  Applying the formula in Theorem \ref{cardano} for $r'$,   $\overline r= \sqrt{p} r'$ gives the formula in the corollary as the only real solution of the  reduced equation.
\end{proof}

\begin{theorem}  The reduced form $x^3+px+q=0$ with $p <0$ has a unique real solution if and only if $\Delta = -(q^2/4+p^3/27) <0$. Moreover, this solution is given by the formula in Corollary \ref{cardcor}.
\end{theorem}

\begin{proof} The proof of the first part of the theorem follows from Tusi Theorem (Theorem \ref{Tusithm}) and Corollary \ref{cardcor}. Next, as in proof of Theorem \ref{cardano} it can be shown that choosing  $s= \sqrt{{q^2}/{4} - {1}/{27}}$, the same formula for $r$ applies. Next, replacing $x$ by $\sqrt{p} x$, as in Corollary \ref{cardcor}, the formula for the reduced equation follows.
\end{proof}

\begin{remark} We contrast the way Cardano's formula is derived in Theorem  \ref{cardano} and in standard approach. The standard approach for  solving $x^3+px+q=0$ takes a solution as $x=u+v$. Substituting for $x$ in the equation, we get $u^3+v^3+(u+v)(3uv+p)+q=0$. Next, it sets $3uv+p=0$, a non-intuitive guess, so as to get a second equation $u^3+v^3=-q$. Next, using the two equations, it recognizes that $u^3$ and $v^3$ must be the solutions to the quadratic equation $x^2+qx-p^3/27=0$. Finally, solving via the quadratic formula gives Cardano's formula.  Theorem  \ref{cardano} does not require solving a  quadratic equation and while it verifies  Cardano's formula, it is not far from an actual derivation as argued next. Suppose $q <0$ ($q >0$ can be argued similarly).
Then $r <-q^{1/3}$, otherwise $r^3+r >-q$.  Thus we may guess as solution $r=(-q - s)^{1/3}$  for some $s >0$. But it can easily be shown this does not lead to a solution. Next we may consider $r$ as the sum of two terms of the form   $(-\gamma_1 q + s )^{1/3}$ and $(-\gamma_2 q - s)^{1/3}$, where  $\gamma_1$  and $\gamma_2$ are positive numbers summing up to one. The natural choice is to take $\gamma_1=\gamma_2$ which leads to the correct solution. In other words in our approach we first guess the form of the solution in terms of $s$, then solve for the unknown value of $s$.
\end{remark}

\begin{remark}
There are other ways to derive Cardano's formula.  For instance, Nickalls \cite{Nickalls} associates parameters to the cubic equation  and shows how they lead to modification of the standard method of Cardano’s solution. He feels the standard approach should be abandoned in favour of the use of parameters which reveal how the algebraic solution is related to the geometry of the cubic. However, from the  pedagogical point of view, considering the normal form instead of the general form is more convenient in remembering the development of the solution, arriving at it directly via simple  algebraic manipulations.  Its extension to the general reduced form is a matter of another algebraic manipulation, shown in the next theorem. Cardano's formula can also be developed via much technical subject of Galois theory, see Janson \cite{Janson}.
\end{remark}

\section{Khayyam's Geometric Solution and Tusi Form} \label{sec5} Omar Khayyam (1048-1131) is well known internationally for his poetry and philosophical thoughts that so compactly emanate through his rubayiat.  His contributions to the study of solution of cubic equations  profoundly influenced the work of future Islamic mathematicians and through them, the mathematicians of Renaissance Europe.  Khayyam's work on cubic equations describes their solution as the intersection of conic sections.  Despite numerous citations of Khayyam's work, when it gets to its description, it always focuses on specific cubic equations, leaving it ambiguous how the geometric method works for the general cubic equation.  Some citations state that Khayyam described geometric ways for solving ``all'' cubic equations while some other citations state the work as applicable to ``some'' cubic equations.

Given that the analysis of Tusi's work has led us to the reduced forms with $p=\pm 1$, it is enough to  describe Khayyam's work for these two types. In fact we will develop Khayyam's approach for these two types in parallel. Consider the positive and negative normal forms
$$x^3+x+q=0, \quad x^3-x+q'=0.$$
Multiplying these by $x$ and letting $y=x^2$ we get
$$y^2+x^2+qx=0, \quad y^2-x^2+q'x=0.$$
Adding $q^2/4$ to both sides of the first equation and $-q'^2/4$ to both sides of the second equation gives
\begin{equation} \label{bothK}
\bigg (x+\frac{q}{2} \bigg )^2+ y^2= \frac{q^2}{4}, \quad \bigg (x-\frac{q'}{2} \bigg )^2-y^2=\frac{{q'}^2}{4}.
\end{equation}

Replacing $x$ by $-x$ changes $x^3 \pm x$ to $-(x^3 \pm x)$ and thus we may assume $q$ and $q'$ have arbitrary signs. We assume $q <0$ and $q' >0$.  The assumption $q< 0$ makes the only root of the positive normal form to be positive. In contrast if $q' >0$, when the negative normal form has a single real root it will be positive and when the equation has three real roots, two of them will be positive (as in Tusi form with $\delta \in (0,1)$).  Thus the first  equation in (\ref{bothK}) describes a circle centered at $(-q/2,0)$ while the second equation in (\ref{bothK}) is a hyperbola centered at $(q'/2,0)$. Both pass through the origin. Since $y=x^2$, solving the normal forms can thus be seen as finding the (nontrivial) intersection of the parabola $y=x^2$ with the circle and the hyperbola, respectively. While Omar Khayyam's actual proof may not have been as simple as above, it completely characterizes Khayyam's work on all nontrivial real cubic equations. The simplicity here is due to the fact that there are only two types of cubic equations, namely the normal forms.  Figures 2 and 3 demonstrate the cases of circle and parabola, respectively.

\begin{remark}  The  above derivation of Khayyam's geometric solution to cubic equations is different from typical proofs, see e.g., Amir-Moez \cite{Amir}.  Our approach is algebraic but gives geometric interpretation at the end. To show the difference with typical derivations of Khayyam's method, consider the equation $x^3+Bx-C=0$, $B >0$, $C >0$.  The quantity $B$ is replaced with $b^2$ and $C$ with $b^2c$. Next, the equation is multiplied by $x$ and divided by $b^2$ to get $x^4/b^2+x^2-cx=0$.  Next, the substitution $y=x^2/b$ is applied, also $c^2/4$ is added to both sides of the equation to complete the square. The resulting equation is $(x-c/2)^2 +y^2=c^2/4$, a circle with center at $(c/2,0)$ of radius $c/2$.  It remains to prove that a nontrivial intersection of the circle with the parabola $y=x^2/b$ is a solution to the original cubic equation. This requires an argument and is proved via similarity of triangles and Pythagorean theorem, see e.g., \cite{Amir}. In contrast our approach bypasses the need for these and works whether we seek nontrivial intersection of the parabola with a circle or with a hyperbola.
\end{remark}

Figure 2 shows three different cases occurring for a positive normal form, when the intersection of parabola and circle is to the right of the center of the circle, at the center, and to the left of the center, corresponding to when $-q =1, 2, 3$, respectively. Figure 3 shows three different cases for the negative normal form corresponding to $q'=0.1, 2/3 \sqrt{3}, 0.7$. These cases are derived from Corollary \ref{cor1}, see  (\ref{corbdd}), from which there are three distinct roots if and only if ${3\sqrt{3}q'}/{4} < {1}/{2}$ (equivalently when $q' < 2/3\sqrt{3}$). When $q'=2/3\sqrt{3}$ it correspond to the case with $\delta =1$ in Tusi form and we see in the figure that the hyperbola and the parabola touch at one point on the right-hand-side corresponding to a double root. There is also a negative root in this case.

\begin{figure}[htpb] \label{FigK1}
\centering
\begin{tikzpicture}[scale=1.]
\begin{axis}[xmin=-1.5, xmax=1.5, ymin=-1., ymax=1.6, samples=200, minor tick num=1, axis lines = middle, legend style={nodes={scale=0.7, transform shape},at={(1.2, 1.1)},anchor=north east}]

 \addplot[variable=t, domain=0:360, red, thick] ({sin(t)/2 + .5}, {cos(t)/2});
  \addplot[variable=t, domain=0:360, green, thick] ({sin(t) + 1.0}, {cos(t)});
   \addplot[variable=t, domain=0:360, blue, thick] ({1.5*sin(t) + 1.5}, {1.5*cos(t)});
  \addplot[black, thick] (x,x*x);
  \addplot[red, mark=*] coordinates{(0.682328, 0.465571)};
    \addplot[green, mark=*] coordinates{(1.0, 1.0)};
    \addplot[blue, mark=*] coordinates{(1.23140, 1.47237)};
  \addplot[black, mark=*] coordinates{(0, 0)};
  %\addlegendentry{$(x - 0.5)^2 + y^2 = 0.25$}
  %\addlegendentry{$y = x^2$}
\end{axis}
\end{tikzpicture}
\caption{The case of $x^3+x+q=0$: Intersection of $y=x^2$ and  $(x+q/2)^2+y^2=q^2/4$, $-q=1, 2, 3$, corresponding to red, green and blue circles, respectively.}
\end{figure}
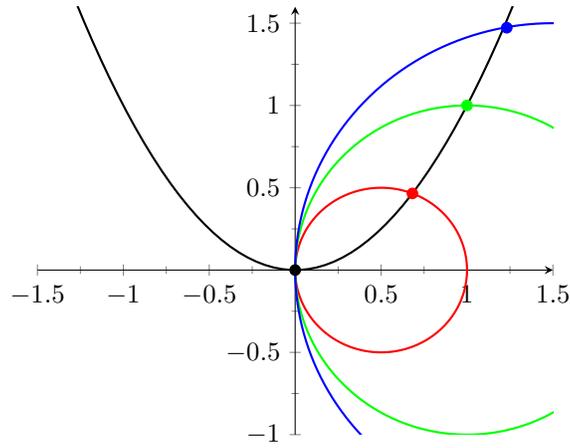

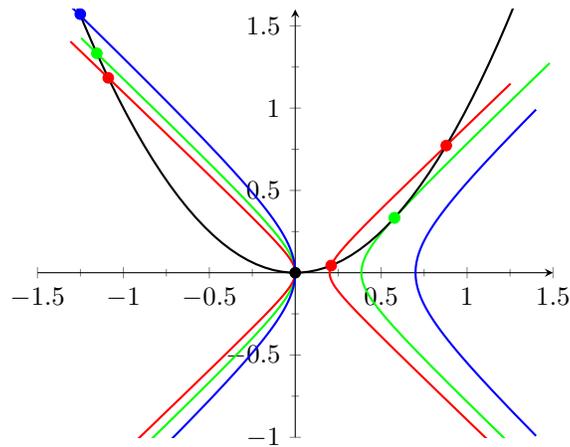
\begin{figure}[htpb] \label{FigK2}
\centering
\begin{tikzpicture}[scale=1.]
\begin{axis}[xmin=-1.5, xmax=1.5, ymin=-1, ymax=1.6, samples=200, minor tick num=1, axis lines = middle, legend style={nodes={scale=0.7, transform shape},at={(1.2, 1.1)},anchor=north east}, restrict x to domain=-1.5:1.5]

  \addplot[variable=t, domain=0:360, green, thick] ({0.19245*sec(t)+0.19245}, {0.037037^.5*tan(t)});
   \addplot[variable=t, domain=0:360, red, thick] ({(sec(t)/10 +0.1)}, {tan(t)/10});
    \addplot[variable=t, domain=0:360, blue, thick] ({(.35*sec(t) + .35)}, {.35*tan(t)});
  \addplot[black, thick] (x,x*x);
  \addplot[green, mark=*] coordinates{(-1.1547, 1.3333)};
   \addplot[red, mark=*] coordinates{(-1.08803, 1.18382)};
    \addplot[blue, mark=*] coordinates{(-1.250803, 1.57082)};
  \addplot[black, mark=*] coordinates{(0, 0)};
  \addplot[green, mark=*] coordinates{(0.577024, 0.332957)};
  \addplot[green, mark=*] coordinates{(0.577676, 0.33371)};
   \addplot[red, mark=*] coordinates{(0.209149, 0.0437432)};
  \addplot[red, mark=*] coordinates{(0.878885, 0.772439)};
  %\addlegendentry{$(x - 0.19245)^2 - y^2 = 0.037037$}
  %\addlegendentry{$y = x^2$}
\end{axis}
\end{tikzpicture}
\caption{The case of $x^3-x+q'=0$: Intersection of $y=x^2$ and hyperbola $(x-q'/2)^2-y^2={q'}^2/4$,  $q'=0.1, 2/3 \sqrt{3}, 0.7$, corresponding to red (three distinct roots), green (double root on the right, also a root on the left), blue (one root on the left) hyperbolas, respectively.}
\end{figure}

Khayyam's geometric description of the solutions does not by itself provide an approximation to a solution. Nor does Tusi's classification of zeros. However, we can actually combine the knowledge from both to devise algorithms that can approximate the real roots of arbitrary real cubic equations to arbitrary precision. We have already seen how Tusi form can help in identifying the number of roots and the intervals containing them. Here we show how Khayyam's classification can give rise to an algorithm. Consider the normal form $x^3+x+q=0$ and assume $q <0$.  The only real solution is the intersection of the parabola $y=x^2$ and the circle $(x+q/2)^2+y^2=q^2/4$. Since the solution is positive, we search for the intersection of the parabola and the half-circle $y=\sqrt{q^2/4-(x+q/2)^2}=\sqrt{-x^2-qx}$.  The only root lies in the interval $(0,-q)$.  The proposed iterative algorithm works as follows:  Depending upon the value of $-q$ being less than one or bigger than one, the  root lies in $(0,-q/2)$ or $(-q/2, -q)$, respectively.
Now assume we have two points $A$, $B$ both belonging to one of the intervals such that at $A$ the parabola lies above the circle and at $B$ it lies below the circle and it is known the root lies in the interval with endpoints $A,B$. The line segment passing through the points $(A,A^2)$ and $(B, B^2)$ intersects the circle at a point with $x$-coordinate $C$.
Since the equation of the line passing through the two points is given by $y=(A+B)x-AB$, $C$ is the positive solution to the equation $\sqrt{-x^2-qx}=(A+B)x-AB$. Equivalently, it is the positive solution to the quadratic equation
$$((A+B)^2+1)x^2 +(q -2AB(A+B))x+A^2B^2=0.$$
This will require the computation of a square-root.  Rather, we could also use instead of $C$ the Newton iterate for solving this quadratic equation with respect to the iterate at $C'=(A+B)/2$, the midpoint of the interval with endpoint $A,B$.  Having computed $C$ or $C'$, depending  on whether the parabola lies over or under the circle at this point, we replace it either with $A$ or $B$ and repeat the process. Such an algorithm could certainly have been conceived by Khayyam, including the use of Newton's iterate for approximation of square roots, known to the Babylonians. The convergence of this procedure would most likely be faster than using only the bisection method. We will however not pursue this here.

For dealing with the negative normal form $x^3-x+q'=0$, $q' >0$, intervals can be derived containing the single root or the three roots as follows. Suppose $r$ is a positive root. Then $r^3-r =-q' <0$.  This implies $r^3 < r$. Since $r >0$, it follows that $r < 1$.  Suppose $r <0$.  If $r \leq -\sqrt{2}$, then $r^2 \geq 2$ implying $r^3/2 \leq  r$. From this we get, $q'= -r^3+r \geq -r^3+r^3/2 =-r^3/2$. Thus $r \geq -{2q'}^{1/3}$. Thus if $r <0$ it must be at least $L=\min \{-\sqrt{2}, -{2q'}^{1/3}\}$.
So any root of the negative normal form must lie in the interval $[L, 1)$. We can thus give a similar algorithm to the one described for positive normal form for approximating roots of negative normal form.

\section{Quadratic Tuis Form} \label{sec6} Here we consider the Tusi methodology for a quadratic equation.

\begin{definition}  The reduced form of a quadratic equation and the corresponding Tusi form are, respectively,
$$
x^2-bx + c =0, \quad b >0, \quad  \alpha^2- \alpha +\frac{\delta}{4}=0, \quad \delta \in \mathbb{R}.$$
\end{definition}

The above forms are analogous to  Tusi formulation for cubic. Unlike the cubic case, there is only one reduced form. Setting $\phi(\alpha)= \alpha - \alpha^2$,
to get its maximum in the interval $[0,1]$, we write $\alpha - \alpha^2=\alpha (1- \alpha)$.  This makes it easy to guess the maximizer to be $\alpha^*=1/2$.  To prove this, set $\alpha^* =1/2 \pm \varepsilon$ for some $\varepsilon \in [0,1/2]$. Then $\alpha^*(1 - \alpha^*)=(1/2-\varepsilon)((1/2+\varepsilon)=1/4 - \varepsilon ^2$. It follows that the maximum $\phi^*$ of $\phi(\alpha)$ over the interval $(- \infty, \infty)$ occurs at $1/2$ and is $1/4$. Figure 4 gives the corresponding graph.

\begin{theorem}

The number of solutions and solutions of the equation  $\alpha^2- \alpha+\frac{\delta}{4}=0$ and $x^2-bx+c=0$ ($b >0$) satisfy, respectively

$$
\begin{cases}
0, \quad  & \text{if ~} \delta > 1\\
1, \quad  \alpha^*=\frac{1}{2}, & \text{if ~} \delta=1\\
2, \quad \frac{1}{2}\pm \frac{\sqrt{1-\delta}}{2}, & \text{if ~} \delta < 1
\end{cases}, \quad  \begin{cases}
0,  \quad  & \text{if ~} c > \frac{b^2}{4}\\
1, \quad  x^*=\frac{b}{2}, & \text{if ~} c=\frac{b^2}{4}\\
2, \quad  \frac{b}{2}\pm \frac{\sqrt{b^2-4c}}{2}, & \text{if ~} c < \frac{b^2}{4}.
\end{cases}
$$
\end{theorem}

\begin{proof}  Analogous to the cubic case, by factoring,
$
\phi(\alpha)-\phi^*= \alpha - \alpha^2- \frac{1}{4}=-(\alpha-\frac{1}{2})^2.
$
Thus to solve the equation $\alpha- \alpha^2- \delta/4 =0$ for a given $\delta \leq 1$, we add and subtract $1/4$ to get
$
(\alpha- \frac{1}{2})^2=\frac{1-\delta}{4}.
$
Then take square-root to solve for $\alpha$.  This proves the first equation. To get the roots of $x^2-bx+c$, we simply multiply the roots of Tusi form by $b$ and replace $\delta$ with $4c/b^2$:
$$
b \bigg ( \frac{1}{2}\pm \frac{\sqrt{1-\delta}}{2} \bigg )= \frac{b}{2} \pm  \frac{b}{2} \sqrt{1- \frac{4c}{b^2}}= \frac{b}{2}\pm \frac{\sqrt{b^2-4c}}{2}.
$$
\end{proof}

\begin{figure}[htpb] \label{Figquad}
\centering
\begin{tikzpicture}[scale=1.2]
\begin{axis}[xmin=-1.5,xmax=1.5,ymin=-1.5,ymax=1.5, samples=200,minor tick num=1,xticklabels = {,,},axis lines = middle, legend pos = south east]
\addplot[blue, thick] (x,x - x*x);
\addplot[red, thick] (x,1/4);
\addplot[black, mark=*] coordinates {(1/2,0)};
\addplot[black, mark=*] coordinates {(1/2,1/4)} node[below = .3cm] {$\alpha^* = \frac{1}{2}$};
\addplot[black, mark=*] coordinates {(1/2,0)} node[above = .5cm] {~~~$\alpha-\alpha^2$};
%\addplot[black, mark=*] coordinates{(-1/3,4/27)} node[below = .3cm] %{$\widehat\alpha_3^* = \frac{-1}{3}$};
%\addplot[black, mark=*] coordinates{(-1/3,0)};
%\addlegendentry{$x - x^2$}
%\addlegendentry{$x^3 + x$}
\end{axis}
\end{tikzpicture}
\caption{Any quadratic equation is reducible to $\alpha-\alpha^2=\delta/4$.}
\end{figure}
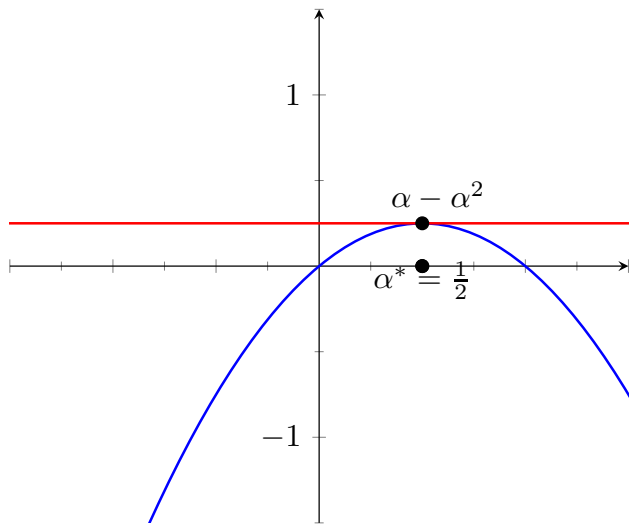

\begin{remark}

Po-Shen Loh \cite{Loh} develops a new way to derive the quadratic formula, also
featured in the New York Times \cite{NYT}.  Here we wish to mention a resemblance between his derivation of the quadratic formula and what Tusi's methodology gives, when applied to a quadratic equation. Loh considers the quadratic equation $x^2+Bx+C=0$, where $B$ and $C$ are complex numbers.  To solve the equation using factorization gives
$$x^2+Bx+C=(x-R)(x-S),$$
where we solve for $R$ and $S$ by recognizing that $R+S=-B$ and $RS=C$.  He writes,\\

``Two numbers sum to $-B$ precisely when their average is $-B/2$, and so it suffices to find two numbers of the form $-B/2\pm z$ which multiply to $C$, where $z$ is a single unknown quantity, because they will automatically have the desired average.''\\

This implies $(-B/2-z)(-B/2+z)=-B^2/4-z^2=C$ from which it follows $z=\pm \sqrt{B^2/4-C}$. Hence the roots are $-B/2 \pm \sqrt{B^2/4-C}$.

We see that when solving quadratic equations with real coefficients there is some resemblance between Loh's approach and what is derivable from the Tusi form of a quadratic, at least in that symmetry of the solutions around $1/2$, the maximizer of $\alpha-\alpha^2$, is easily witnessed. Indeed when solving a quadratic equation with real coefficients, from pedagogical point of view  it is constructive to reduce it into the Tusi form.  Simply by keeping in mind the graph of $\alpha-\alpha^2$ and by reducing the particular equation to the form  $\alpha- \alpha^2=\delta/4$, one immediately gets a sense of the solvability of the equation and the symmetry of the solutions around $1/2$.  In summary, understanding the shape of one quadratic function is enough to grasp the way we solve all quadratic equations.
\end{remark}

\section{Generalized Tusi Form} \label{sec7}

By a {\it generalized Tusi form} we mean the polynomial  equation,
$p_n(\alpha)=\alpha^n-\alpha^{n-1}+c=0$, $n \geq 2$. The following is obvious and analogous to the case of $n=3$. Letting $\phi_n(\alpha) = \alpha^{n-1}- \alpha^n$,
a real number $r$ satisfies $p_n(r)=0$ if and only if  $\phi_n(r)=c$.

\begin{lemma}  \label{lem2}

$$
\phi_n(\alpha)=
\begin{cases}
>0, & \text{if $n$ is odd and ~} \alpha < 0\\
<0, & \text{if $n$ is even and ~} \alpha < 0\\
0, & \text{if ~} \alpha=0\\
>0 , & \text{if ~} 0 < \alpha < 1 \\
0, & \text{if ~} \alpha =1\\
< 0, & \text{if~} \alpha >1
\end{cases},
\quad
\phi_n(\alpha)
\begin{cases}
\rightarrow -\infty, & \text{as~} \alpha  \rightarrow +\infty\\
\rightarrow +\infty,  & \text{as~} \alpha \rightarrow -\infty~ \text{if n is odd}\\
\rightarrow -\infty,  & \text{as~} \alpha \rightarrow -\infty~  \text{if n is even}.
\end{cases}
$$
\end{lemma}

\begin{proof}
When $\alpha <0$, $\phi_n(\alpha)$ is positive for $n$ odd and negative for $n$ even. On  the other hand, $\phi_n(\alpha)=0$ for $\alpha=0$ and $\alpha=1$,  $\phi_n(\alpha) >0$ if $\alpha \in (0,1)$ and $\phi_n(\alpha) <0$ for $\alpha >1$.  This verifies the first part.  The proof of the second part follows from the equation $\phi_n(\alpha)=\alpha^{n-1}(1-\alpha)$.
\end{proof}

Analogous to the case of $n=2$ and $n=3$, without using differentiation one can guess the location of the maximizer, $\alpha_n^*$ of $\phi_n$ on $(0,1)$. In fact recalling that $\alpha_2^*=1/2$ and $\alpha_3^*=2/3$, it is straightforward to guess $\alpha_n^*$ to be $(n-1)/n$. Analogous to the case of $n=3$, Lemma \ref{lem3} verifies the correctness of this guess. Subsequently, the theorem uses the lemma to classify the number of zeros of $\phi_n$ and containing intervals.

\begin{lemma}   \label{lem3} $\phi_n^* = \max \{\phi_n(\alpha): \alpha \in [0, \infty)\} = \phi(\alpha_n^*)$, where
$\alpha _n^*= \frac{n-1}{n}$. Hence $\phi_n^*=\frac{(n-1)^{n-1}}{n^n}$. Moreover,

\begin{equation} \label{lem1eq2}
\phi_n(\alpha)
\begin{cases}
\nearrow   \text{on~} (-\infty, 0], & \text{if n is even}\\
\searrow   \text{on~} (-\infty, 0], & \text{if n is odd}\\
\nearrow   \text{on~} [0, \alpha_n^*]\\
\searrow  \text{on~} [\alpha_n^*, \infty).
\end{cases}
\end{equation}
\end{lemma}

\begin{proof} We first prove  the maximum of
$\alpha^{n-1}-\alpha^n - \phi_n^*$ on $[0,\infty)$ is zero. Dividing by
 $(\alpha - \alpha_n^*)$ we get the factorization
\begin{equation} \label{eq11}
\alpha^{n-1}-\alpha^n - \phi_n^*=(\alpha_n^*- \alpha ) U_n(\alpha),
\end{equation}
where
\begin{equation}
U_n(\alpha)=\alpha^{n-1}- \frac{1}{n}\alpha^{n-2}-\frac{n-1}{n^2}\alpha^{n-3} - \cdots - \frac{(n-1)^{n-3}}{n^{n-2}}\alpha- \frac{(n-1)^{n-2}}{n^{n-1}}.
\end{equation}
The factorization is easily verifiable by multiplying the terms and simplifying. We claim $U_n(\alpha_n^*)=0$.  The first term of $U_n(\alpha_n^*)$  is $(n-1)^{n-1}/n^{n-1}$ while each of the remaining $n-1$ terms is $-(n-1)^{n-2}/n^{n-1}$. Thus the $n$ terms sum to zero. Suppose $\alpha > \alpha_n^*$. Then for each $i=2, \dots, n$,
\begin{equation}
-\frac{(n-1)^{i-2}}{n^{i-1}}\alpha^{n-i} \leq -\frac{(n-1)^{n-2}}{n^{n-1}}.
\end{equation}
It follows that
\begin{equation} \label{eq13}
U_n(\alpha) \geq
\alpha^{n-1} - \frac{(n-1)^{n-1}}{n^{n-1}} =\alpha^{n-1}- {\alpha_n^*}^{n-1}  >0.
\end{equation}
Thus from (\ref{eq11}) and (\ref{eq13}) it follows that $\alpha > \alpha_n^*$ implies $\alpha^{n-1}-\alpha^n - \phi_n^* <0$. Similarly,  $\alpha \in [0, \alpha_n^*)$ implies $\alpha^{n-1}-\alpha^n - \phi_n^* >0$. Hence the proof that $\alpha_n^*$ is the maximizer of $\phi_n(\alpha)$ on $[0, \infty)$.  To prove monotonicity we factor $U_n(\alpha)$.  It is straightforward to verify
\begin{equation}
U_n(\alpha)= (\alpha - \alpha_n^*) V_n(\alpha),
\end{equation}
\begin{equation}
V_n(\alpha)=\alpha^{n-2}+ \frac{(n-1)^0(n-2)}{n}\alpha^{n-3}+\frac{(n-1)(n-3)}{n^2}\alpha^{n-4} +
\cdots + \frac{(n-1)^{n-4} \times 2}{n^{n-3}}\alpha+ \frac{(n-1)^{n-3}}{n^{n-2}}.
\end{equation}
Thus
\begin{equation} \label{eqfactor}
 \alpha^{n-1}-\alpha^n- \phi_n^* = -(\alpha - \alpha_n^*)^2 V_n(\alpha).
\end{equation}
All coefficients in $V_n(\alpha)$ are positive so that $0 \leq \alpha_1 < \alpha_2$ implies $V_n(\alpha_1) < V_n(\alpha_2)$.  When $0 \leq \alpha_1 <\alpha_2 \leq \alpha_n^*$,
\begin{equation}
-(\alpha_1 - \alpha_n^*)^2 < - (\alpha_2 - \alpha_n^*)^2.
\end{equation}
From these and (\ref{eqfactor}) it follows that $\phi_n(\alpha)$ is strictly increasing on $[0, \alpha_n^*]$.  When
 $\alpha_n^* \leq \alpha_1 <\alpha_2$,
\begin{equation}
 (\alpha_1 - \alpha_n^*)^2 <  (\alpha_2-\alpha_n^* )^2.
 \end{equation}
Thus
\begin{equation}
 (\alpha_1-\alpha_n^*)^2 V_n(\alpha_1) <  (\alpha_2 -\alpha_n^*)^2 V_n(\alpha_2).
 \end{equation}
This and (\ref{eqfactor}) imply $\phi_n(\alpha)$ is strictly decreasing on $[\alpha_n^*,\infty)$.  To prove the first two parts of (\ref{lem1eq2}), from Lemma \ref{lem2},  $|\phi_n(\alpha)|$ approaches infinity as $\alpha$ approaches $-\infty$. Suppose $\alpha_1 < \alpha_2 \leq 0$. Then
\begin{equation}
|\alpha_1^{n-1}| > |\alpha_2^{n-1}|, \quad (1-\alpha_1) > (1-\alpha_2).
\end{equation}
Thus $|\phi_n(\alpha_1)| >  |\phi_n(\alpha_2)|$ so that
on the interval $(-\infty, 0]$ for any $\alpha_1 \not = \alpha_2$,  $\phi_n(\alpha_1) \not = \phi_n(\alpha_2)$. This implies strict monotonicity.
\end{proof}

\begin{prop} \label{prop1} {\rm (Intermediate Value Theorem)} If $\phi_n(\alpha_1) = c_1$ for some $\alpha_1$ and $c_1$ and
$\phi_n(\alpha_2) = c_2$ for some $\alpha_2$ and $c_2$, $c_1 < c_2$, then for any $c' \in (c_1,c_2)$, there exists $\alpha'$ such that $\phi_n(\alpha')=c'$. \qed
\end{prop}

The following theorem characterizes the solutions of $\phi_n(\alpha)=\delta \phi_n^*$ which depend on $n$ as a parameter and $\delta$.  Using this  characterization the theorem can then characterize roots of  $p_n(\alpha)=0$ (omitted).

\begin{theorem} \label{thm1} Given a real $\delta$, the number of solutions and the solutions  of the equation $\phi_n(\alpha)=\delta \phi_n^*$ satisfy
$$
\begin{cases}
1, \quad  \overline \alpha \in (1, \infty),& \text{if $n$ odd, ~} \delta < 0\\
2, \quad \underline \alpha_1 \in (-\infty,0), \quad \overline \alpha_2 \in (1, \infty), & \text{if $n$ even, ~} \delta < 0\\
2, \quad \alpha_1 =0, \quad \alpha_2=1, & \text{if ~} \delta = 0\\
3, \quad  0 < \alpha_1 < \alpha_n^* < \alpha_2 < 1, \quad
\alpha_3 \in (-\alpha_n^*,0),& \text{if $n$ odd,~} 0 < \delta < 1\\
2, \quad  0 < \alpha_1 < \alpha_n^* < \alpha_2 < 1, & \text{if $n$ even~} 0 <  \delta < 1\\
2, \quad  \alpha_n^*, \quad  \underline \alpha_n^* \in [-\alpha_n^*,0), & \text{if $n$  odd,~} \delta=1\\
1, \quad \alpha_n^*, & \text{if $n$ even, ~} \delta=1\\
1, \quad \underline \alpha \in (-\infty, -\alpha_n^*),   & \text{if $n$  odd, ~} \delta >1\\
0, \quad  & \text{if $n$ even, ~} \delta > 1.
\end{cases}
$$
\end{theorem}

\begin{proof} The proof of the theorem uses Lemmas \ref{lem2}, \ref{lem3} as well as Proposition \ref{prop1}.  The proof of the first item  follows from Lemma \ref{lem2} and Proposition \ref{prop1}:  $\phi_n$ is negative on the interval $(1, \infty)$ and it approaches $-\infty$ as $\alpha$ approaches $\infty$. Thus by  Proposition \ref{prop1} the corresponding equation has a root and by strict monotonicity it is a unique root. Analogous argument applies to the second statement. The proof of the third statement uses Lemma \ref{lem2}, Lemma \ref{lem3} and Proposition \ref{prop1} to conclude there is a root $\alpha_1$ in the interval $[0, \alpha_n^*)$ and another $\alpha_2$ in $(\alpha_n^*,1)$. To prove the existence of the third root $\alpha_3$ for $n$ odd,  since  $\phi_n(0)=0$ and  $\phi_n(\alpha)=\alpha^{n-1}(1-\alpha)$, we have
$$
\phi_n(-\alpha_n^*) =\frac{(n-1)^{n-1}}{n^{n-1}} \bigg ( 1+ \frac{n-1}{n} \bigg )=
\frac{(n-1)^{n-1}}{n^{n}}(2n-1) > \phi_n^*.
$$
Hence by Proposition \ref{prop1} there is a root in the indicated interval.
The proof of the forth statement is similar to the previous one but there is no third root. The proof of the fifth statement also uses Lemma \ref{lem2} and the factorization in Lemma \ref{lem3}, showing $\alpha_n^*$ is a double root.
When $n$ is odd by Proposition \ref{prop1} there is another root in the indicated interval (sixth statement) but when $n$ is even there are no other roots (seventh statement). The proof of eight and ninth statements follow analogously.
\end{proof}

Figure 5 shows the graph of $\alpha^{n-1}-\alpha^n$ for the first few values of $n$. We see the maximum value drops as $n$ increases and the maximizer moves toward one.  Using Theorem \ref{thm1} an exact formula for solutions of $\phi_n(\alpha)=\delta \phi_n^*$ directly gives an exact formula for roots of $p_n(\alpha)$.  To find  positive roots of $p_n$ we only need to look at the range of $\alpha, \delta \in [0,1]$.  If $\delta$ is outside of the range $[0,1]$, we conclude there is no positive solution.  As in the quadratic or cubic case, using binary search or in combination with Newton's method we can approximate the roots to any precision.

\begin{figure}[htpb] \label{Fig4}
\centering
\begin{tikzpicture}[scale=1.]
\begin{axis}[xmin=-0.3, xmax=1.2, ymin=-0.3, ymax=0.3, samples=200, minor tick num=1, axis lines = middle, legend style={nodes={scale=0.7, transform shape},at={(1.2, 1.1)},anchor=north east}, restrict x to domain=-1.5:1.5]

  \addplot[blue, thick] (x, {-(x-1)*x});
  \addplot[green, thick] (x, {-(x-1)*x*x});
  \addplot[red, thick] (x, {-(x-1)*x*x*x});
  \addplot[purple, thick] (x, {-(x-1)*x*x*x*x});
  %\addlegendentry{$-(x-1)x$}
  %\addlegendentry{$-(x-1)x^2$}
  %\addlegendentry{$-(x-1)x^3$}
  %\addlegendentry{$-(x-1)x^4$}

\end{axis}
\end{tikzpicture}
\caption{Graphs of $\alpha^{n-1}-\alpha^n$, $n=2,3,4,5$, blue, green, red, purple, respectively.}
\end{figure}
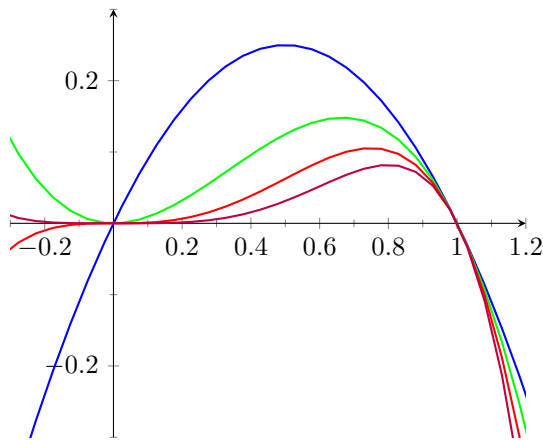

\section*{Concluding Remarks} In this article we have examined the work of Sharaf al-Din Tusi of the 12th century on cubic equations, not only from a historical point of view but from a contemporary point of view in terms of its connections to the well known Cardano's formula and its application in approximation of real zeros.  Using Tusi's methodology we characterized all real cubic equations by reducing them into two disjoint forms, a Tusi form with  $\delta \in (0,1)$  and  a reduced form with $p=1$. This characterization not only helps determine the number of real zeros but tight  containing intervals when there are three roots. This in turn  allows their numerical approximation using the bisection and/or Newton methods, bypassing the use a Cardano's formula expressed in complex numbers which itself require numerical approximation. We showed the discriminant present in Cardano's formula, is inherent in Tusi form. In fact we gave a novel proof of
Cardano's formula for a reduced form with $p >0$.  Gaining insights from the study of Tusi form, we gave  a concise derivation of Khayyam's geometric solution. Characterization  of cubic equations into two canonical forms simplifies computing their solution. In fact using Tusi's methodology we showed solving a quadratic equation is reducible  to a quadratic Tusi form from which the quadratic formula is derivable trivially.  Such canonical classifications for quartic, quintic and general  polynomial equations may also prove to be useful.  We also extended Tusi forms to arbitrary degree, showing analogous properties. We believe our results based on the study of Tusi form complement previous findings on Tusi's work and reveal further facts on history, mathematics and pedagogy in solving cubic equations.
Hence the results should be of interest to general mathematical community, historians, educators and students.

\section*{Acknowledgements}
The authors wish to thank Professor Roshdi Rashed for some references.

\bigskip

\end{document}